\documentclass[11pt,reqno]{amsart}
\usepackage{amsmath,amsfonts,amssymb,amsthm,amscd,comment,euscript}
\usepackage[all]{xy}
\usepackage{graphicx}
\usepackage{mathptmx}
\usepackage{enumerate} %allows to change enumeration items easily
\usepackage[colorlinks=true]{hyperref} %to generate pdf with links
\usepackage{enumitem}

\xymatrixcolsep{1.9pc}                          % Adjust size of diagrams.
\xymatrixrowsep{1.9pc}
\newdir{ >}{{}*!/-5pt/\dir{>}}                  % Make better tailed arrows
 \calclayout

\swapnumbers
\theoremstyle{plain}
\newtheorem{lem}{Lemma}[section]
\newtheorem{cor}[lem]{Corollary}
\newtheorem{prop}[lem]{Proposition}
\newtheorem{thm}[lem]{Theorem}

\newtheorem*{conj}{Conjecture}

\theoremstyle{definition}
\newtheorem{ex}[lem]{Example}
\newtheorem{rem}[lem]{Remark}
\newtheorem{dfn}[lem]{Definition}

\renewcommand{\phi}{\varphi}
\renewcommand{\leq}{\leqslant}
\renewcommand{\geq}{\geqslant}
\renewcommand{\epsilon}{\varepsilon}

\renewcommand{\kappa}{\varkappa}

\DeclareMathOperator{\zar}{Zar}

\DeclareMathOperator{\spec}{Spec}
\DeclareMathOperator{\ind}{ind} 

 \DeclareMathOperator{\charr}{char}

\DeclareMathOperator{\Hom}{Hom} 
 \DeclareMathOperator{\id}{id}

 \DeclareMathOperator{\colim}{colim}
 
 \DeclareMathOperator{\Ab}{Ab}
 
 \DeclareMathOperator{\kr}{Ker}
 \DeclareMathOperator{\im}{Im}
\DeclareMathOperator{\coker}{Coker} \DeclareMathOperator{\nis}{Nis}

\newcommand {\lp}{\varinjlim}

\newcommand{\lra}[1]{\bl{#1}\longrightarrow\relax}
\newcommand{\bl}[1]{\buildrel #1\over}
\newcommand{\cc}{\mathcal}
\newcommand{\bb}{\mathbb}

\newcommand{\wh}{\widehat}
\newcommand{\wt}{\widetilde}

\newcommand{\gmpn}{\bb G_m^{\wedge n}}
\newcommand{\uhom}{\underline{\Hom}}

\begin{document}

\footskip30pt

%\baselineskip=1.5\baselineskip

\title{Semilocal Milnor K-theory}
\author{Grigory Garkusha}
\address{Department of Mathematics, Swansea University, Fabian Way, Swansea SA1 8EN, United Kingdom}
\email{g.garkusha@swansea.ac.uk}
%\thanks{}

\begin{abstract}
In this paper, semilocal Milnor $K$-theory of fields is introduced and studied. A strongly convergent spectral
sequence relating semilocal Milnor $K$-theory to semilocal motivic cohomology is constructed.
In weight 2, the motivic cohomology groups $H^p_{\zar}(k,\bb Z(2))$, $p\leq 1$, are computed
as semilocal Milnor $K$-theory groups $\wh{K}^M_{2,3-p}(k)$. The following applications are given:
(i) several criteria for the Beilinson--Soul\'e Vanishing Conjecture;
(ii) computation of $K_4$ of a field; (iii) the Beilinson conjecture for rational $K$-theory of fields of prime characteristic
is shown to be equivalent to vanishing of rational semilocal Milnor $K$-theory. 
\end{abstract}

\thanks{Supported by EPSRC grant EP/W012030/1}



%\date{15 April 2020}

\maketitle

\thispagestyle{empty} \pagestyle{plain}

\newdir{ >}{{}*!/-6pt/@{>}} %this command is to define the arrow of the type \ar@{ >->} (spacing of arrows is needed sometimes)

\tableofcontents

\section{Introduction}

It is a classical fact of algebraic $K$-theory of fields that Milnor $K$-groups $K_0^M,K_1^M,K_2^M$ agree
with Quillen's $K_0,K_1,K_2$. However, $K_n^M$ is only a small piece of Quillen's $K_n$
for $n\geq 3$. 
A key technical tool to make computations in algebraic $K$-theory is the motivic spectral sequence
   $$E_{p,q}^2=H_{\zar}^{q-p}(k,\bb Z(q))\Longrightarrow K_{p+q}(k)$$
relating algebraic $K$-theory to motivic cohomology (see, e.g.,~\cite{FS}).

By well-known theorems of Nesterenko--Suslin~\cite{NS} and Totaro~\cite{Tot} the Milnor
$K$-theory ring $K_*^M(k)$ is isomorphic to the ring $\bigoplus H^n_{\zar}(k,\bb Z(n))$. 
We also know that $H^1_{\zar}(k,\bb Z(2))=K_3^{\ind}(k)$, where $K_3^{\ind}(k)$ is the 
indecomposable $K$-theory of $k$. The other motivic cohomology groups are a complete mystery.
Their computation, and hence computation of algebraic $K$-theory, is one of the hardest problems in the field and several
outstanding conjectures are related to this problem. For instance, the celebrated Beilinson--Soul\'e Vanishing Conjecture
states that all motivic cohomology groups $H^p_{\zar}(k,\bb Z(q))$ vanish for $p\leq 0$ and $q>0$~\cite[\S3]{SV}. 
The finite coefficient motivic cohomology groups are much better understood due to the norm residue isomorphism 
theorem (formerly known as Milnor/Bloch--Kato Conjectures) relating the motivic cohomology to \'etale cohomology~\cite{V10}. 
The mysteries (such as the Beilinson--Soul\'e Vanishing Conjecture) mostly 
surround the $K$-theory and motivic cohomology with rational coefficients.
In positive characteristic,
the Beilinson conjecture states that Milnor $K$-theory and Quillen $K$-theory agree rationally:
   $$K_n^M(k)_{\bb Q}\lra{\cong} K_n(k)_{\bb Q}.$$
%Both conjectures are far from being solved.
   
As we have mentioned above, Milnor $K$-theory is isomorphic to the motivic 
cohomology diagonal $\bigoplus H^n_{\zar}(k,\bb Z(n))$. The main purpose of this paper is
to introduce and investigate ``semilocal Milnor $K$-theory of fields". We show that it is precisely
related to motivic cohomology outside the diagonal. An advantage of the theory 
is that it is defined in elementary terms
whereas the motivic complexes  
are sophisticated and enormously hard for computations. All the definitions of 
the motivic complexes are strictly geometric, whereas the definition of 
semilocal Milnor $K$-theory is strictly algebraic (whenever the base field is infinite
-- see Remark~\ref{naive}).

By definition, semilocal Milnor $K$-theory of a field $k$ consists of bigraded Abelian groups
$\wh{K}_{n,m}^M(k)$, $m,n\geq 0$ (see Definition~\ref{maindef}). 
Precisely, let $\wh{\Delta}^\bullet_k$ be the cosimplicial scheme, where each
$\wh{\Delta}^\ell_k$ is the semilocalization of the standard affine scheme 
${\Delta}^\ell_k$ at its vertices $v_0,\ldots,v_\ell$ (see Definition~\ref{deltahat}).
Let $\cc K_n^M$ be the Zariski sheaf of Milnor $K$-theory in degree $n\geq 0$. 
Semilocal Milnor $K$-theory complex is 
the chain complex $\cc K_n^M(\wh{\Delta}^\bullet_k)$ and
   $$\wh{K}_{n,m}^M(k):=H_m(\cc K_n^M(\wh{\Delta}^\bullet_k)).$$
If $k$ is infinite, the complex $\cc K_n^M(\wh{\Delta}^\bullet_k)$ is 
defined naively in terms of generators and relations (see Remark~\ref{naive}). So
semilocal Milnor $K$-theory groups are defined as homology groups 
of complexes defined by generators and relations.

The main result of the paper, Theorem~\ref{glavnaya}, says that there is a strongly convergent spectral sequence
relating semilocal Milnor $K$-theory to semilocal motivic cohomology
   $$E^2_{pq}=H_{p}(H^{n-1-q}_{\zar}(\wh{\Delta}^\bullet_{k},\tau_{<n}\bb Z(n)))\Longrightarrow\wh{K}_{n,p+q+2}^M(k).$$
Here $\tau_{<n}\bb Z(n)$ is the truncation complex of $\bb Z(n)$ for degrees smaller than $n$ and
$H^{n-1-q}_{\zar}(\wh{\Delta}^\bullet_{k},\tau_{<n}\bb Z(n))$ is the chain complex obtained by the evaluation
of the $(n-1-q)$-th Zariski cohomology sheaf of the complex $\tau_{<n}\bb Z(n)$ at the cosimplicial semilocal 
scheme $\wh{\Delta}^\bullet_{k}$.
Moreover, if $n=2$, then the spectral sequence above collapses, and
hence for any $p\leq 1$ there is an isomorphism $H^p_{\zar}(k,\bb
Z(2))=\wh{K}_{2,3-p}^M(k)$. Thus semilocal Milnor $K$-theory is related to motivic cohomology exactly
outside the diagonal $\oplus H^n(k,\bb Z(n))$ in contrast with the classical Milnor $K$-theory. 
Another important and strong property of semilocal Milnor $K$-theory that distinguishes it, say, from motivic
cohomology and Quillen's $K$-theory is that it is invariant under purely transcendental extensions
(see Theorem~\ref{trancs}). The spectral sequence of Theorem~\ref{glavnaya} also implies that 
plenty of information is removed from motivic cohomology groups to get semilocal Milnor $K$-theory groups
(see Corollary~\ref{coker} as well).

Various applications of semilocal Milnor $K$-theory are given in the paper. First, several criteria for the 
Beilinson--Soul\'e Vanishing Conjecture are established in Theorem~\ref{bscriteria}. We next pass to computation
of $K_4$ of a field. The group $K_3$ was actively investigated in the 80-s -- see
Levine~\cite{L87,L89}, Merkurjev and Suslin~\cite{MS,Suslin90} (it is worth mentioning that
semilocal PID-s play an important role in their analysis). Recall that $K_3(k)$ fits into an exact sequence
   $$0\to K_3^M(k)\to K_3(k)\to K_3^{\ind}(k)\to 0.$$
%where $K_3^{\ind}(k)$ is the indecomposable $K$-theory of $k$. 
We compute the group $K_3^{\ind}(k)$ as $\wh{K}_{2,2}^{M}(k)$
 in Corollary~\ref{Kind}, so that $K_3(k)$ is fully determined by Milnor $K$-theory
and semilocal Milnor $K$-theory. The latter computation is of independent interest. 
Similarly to $K_3(k)$ we show in Theorem~\ref{K4} that $K_4(k)$
is also fully determined by Milnor $K$-theory and semilocal Milnor $K$-theory.
These are somewhat surprising results.
As a whole, the relationship between the semilocal Milnor $K$-theory 
$\wh{K}_{*,*}^M(k)$ and algebraic $K$-theory $K_*(k)$ is of great interest.
Indeed, semilocal Milnor $K$-theory groups are homology groups of complexes
defined naively by generators and relations (if $k$ is infinite), 
whereas $K_*(k)$ is defined as homotopy groups of some 
infinite-dimensional space and thus seemingly very inaccessible to computations.

%Theorem~\ref{K4} also implies that Milnor $K$-theory sheaves $\cc{K}^M_2,\cc{K}^M_3,\cc{K}^M_4$ 
%contain all the necessary information to recover
%Quillen's $K_3$ and $K_4$. This is a somewhat surprising result.
%Basing on the above computations, the author cautiously conjectures that Milnor $K$-theory sheaves
%$\cc{K}^M_{n\geq 2}$ contain all the necessary information to recover 
%higher algebraic $K$-theory groups of fields.

Another application is given for the Beilinson conjecture on the rational $K$-theory of fields of prime characteristic. 
Namely, it is shown in Theorem~\ref{beilinson} that this conjecture is equivalent to vanishing
of rational semilocal Milnor $K$-theory.
%, which is much more accessible for computations. 
Also, vanishing of rational semilocal Milnor $K$-theory is shown to be a necessary condition for Parshin's conjecture
(see Theorem~\ref{parshin}).

In the final Section~\ref{concl} it is shown that in contrast with motivic cohomology groups with mod 2 coefficients, 
semilocal Milnor $K$-theory groups with $\mathbb Z/2$-coefficients $\wh{K}^M_{n,*}(k,\mathbb Z/2)$ 
are zero for any $n>1$.
This is another property of semilocal Milnor $K$-theory distinguishing it with the classical Milnor 
$K$-theory/motivic cohomology of fields. We also raise a conjecture in this section on 
rational contractibility of the logarithmic de 
Rham--Witt sheaves $W_r\Omega^n_{\log}$.

The author would like to thank Daniil Rudenko and Matthias Wendt for numerous
discussions on the Beilinson--Soul\'e Vanishing Conjecture. He also thanks Jean Fasel
and the anonymous referees for helpful comments. 

\section{Preliminaries}\label{deltahat}

Throughout the paper we denote by $Sm/k$ the category of smooth
separated schemes of finite type over a field $k$. By a smooth
semilocal scheme over $k$ we shall mean a $k$-scheme $W$ for which
there exists a smooth affine scheme $X\in Sm/k$ and a finite set
$x_1,\ldots,x_n$ of points of $X$ such that $W$ is the
inverse limit of open neighborhoods of this set. We will deal with both 
complexes for which the differential has degree $-1$ (chain complexes) 
and those for which the differential has degree $+1$ (cochain complexes).
By the homological shift of a chain complex $A$ we mean the chain complex $A[1]$
with $A[1]_n=A_{n+1}$ and differential $-d^A$. The author should stress
that the machinery of motivic homotopy theory is not used here (except
the proofs of Theorem~\ref{trancs} and Proposition~\ref{fdelta}), because
we often work with non-$\bb A^1$-invariant (pre)-sheaves having no transfers.

Following~\cite[\S2]{Suslin}, for any presheaf $\cc F : Sm/k\to\Ab$, let $\wt{C}_1\cc F$ denote
the following presheaf:
   $$\wt{C}_1\cc F(X)=\mathop{\lim_{\longrightarrow}}_{X\times\{0,1\}\subset U\subset X\times\bb A^1}\cc F(U).$$
There are two obvious presheaf homomorphisms (given by restrictions
to $X\times 0$ and $X\times 1$ respectively)
$i_0^*,i_1^*:\wt{C}_1\cc F\to\cc F$.

\begin{dfn}[\cite{Suslin, SV}]\label{ratcontr}
A presheaf $\cc F$ is said to be {\it rationally contractible\/} if there
exists a presheaf homomorphism $s:\cc F\to\wt{C}_1\cc F$ such that $i_0^*s=0$ and $i_1^*s=\id$.
\end{dfn}

\begin{ex}\label{primery}
(1) Given $n,l>0$, the Zariski sheaves with transfers $\mathbb Z_{tr}(\bb G_m^{\wedge n}):=Cor(-,\bb G_m^{\wedge n})$
and $\mathbb Z_{tr}(\bb G_m^{\wedge n})/l$ defined in~\cite[Section~3]{SV} are rationally contractible by~\cite[9.6]{SV}.

(2) Let $k$ be a perfect field of characteristic not 2. Then the presheaf with
Milnor--Witt correspondences 
$\wt{\mathbb Z}(\bb G_m^{\times n})/\wt{\mathbb Z}((1,\ldots,1)):=\wt{Cor}(-,\bb G_m^{\times n})/\wt{Cor}(-,(1,\ldots,1))$
in the sense of~\cite{CF} is
rationally contractible by~\cite[2.5]{BF}, and hence its direct summand 
$\wt{\mathbb Z}(\bb G_m^{\wedge n}):=\wt{Cor}(-,\bb G_m^{\wedge n})$ is (see the proof of~\cite[9.6]{SV} as well). 
$\wt{\mathbb Z}(\bb G_m^{\wedge n})$ is a Zariski sheaf by~\cite[5.2.4]{CF}.
\end{ex}

\begin{dfn}\label{deltahat}
Given a field $k$ and $\ell\geq 0$, let $\cc O(\ell)_{k,v}$ denote
the semilocal ring of the set $v=\{v_0=(1,0,\ldots,0),\ldots,v_n=(0,\ldots,0,1)\}$ of vertices of
$\Delta^\ell_{k}=\spec (k[t_0,\ldots,t_\ell]/(t_0+\cdots+t_\ell-1))$
and set
   $$\wh{\Delta}^\ell_{k}:=\spec\cc O(\ell)_{k,v}.$$
Then $\ell\mapsto\wh{\Delta}^\ell_{k}$ is a cosimplicial semilocal
subscheme of ${\Delta}^\bullet_{k}$.
\end{dfn}

%Let $E$ be an $\bb A^1$-invariant and Nisnevich excisive $S^1$-spectrum and let $s_0(E)$ be its
%zeroth slice. The following lemma computes its sections $s_0(E)(Y)$, $Y\in Sm/k$, where $s_0(E)(Y)$ is,
%by definition, the value of a motivically fibrant replacement of $s_0(E)$ at $Y$.

%\begin{lem}(\cite[2.2.6]{KL})\label{kls0}
%For $Y\in Sm/k$, $s_0(E)(Y)$ is weakly equivalent to the total spectrum $E(\wh{\Delta}^\bullet_{k(Y)})$
%of the simplicial spectrum $\ell\mapsto E(\wh{\Delta}^\ell_{k(Y)})$.
%\end{lem}

%\begin{cor}(\cite[3.4]{GP5})\label{kls0cor}
%$s_0(E)=0$ in $SH_{S^1}(k)$ if and only if the total spectrum $E(\wh{\Delta}^\bullet_{K})=0$ in $SH$
%for all finitely generated fields $K/k$.
%\end{cor}

\begin{prop}[Suslin~\cite{Suslin}]\label{facts}
The following statements are true:
\begin{itemize}
\item[$(1)$] Let $\cc F:Sm/k\to\Ab$ be a rationally contractible presheaf.
Then the presheaf $C_n(\cc F)=\uhom({\Delta}^n_{k},\cc F)$ is also
rationally contractible.
\item[$(2)$] Assume that the presheaf $\cc F$ is rationally contractible. Then
the complex $\cc F(\wh{\Delta}^\bullet_{k})$ is contractible,
and hence acyclic.
\end{itemize}
\end{prop}

%Denote by $\cc D(Shv_k)$ the derived category of Nisnevich sheaves
%of Abelian groups on $Sm/k$.

If $\cc F^\bullet$ is a cochain complex, then the canonical truncation $\tau_{<0}\cc F^\bullet$ of $\cc F^\bullet$
has the property that $H^i(\tau_{<0}\cc F^\bullet)=H^i(\cc F^\bullet)$ for $i<0$ and $H^i(\tau_{<0}\cc F^\bullet)=0$ if $i\geq 0$.
If $\cc F^\bullet$ is a cochain complex of presheaves, we will write $\cc H^{-q}$ to denote the
$-q^{\textrm th}$ cohomology presheaf of the complex $\cc F^\bullet$.

By a tower in a triangulated category $\cc T$ we mean a sequence of maps 
   $$\cdots\xrightarrow{f_{q+2}}X_{q+1}\xrightarrow{f_{q+1}}X_q\xrightarrow{f_{q}}\cdots\xrightarrow{f_{1}}X_0$$
in $\cc T$. The $q$-th layer of the tower is an object $B_q\in\cc T$ fitting in a cofiber sequence
$X_{q+1}\xrightarrow{f_{q+1}}X_q\to B_q\to \Sigma X_{q+1}$ in $\cc T$.

The following result plays an important role in our analysis. It says that the zeroth cohomology of a non-positive cochain
complex of rationally contractible presheaves evaluated at $\wh{\Delta}^\bullet_{k}$ is recovered, up to homology,
from negative cohomology evaluated at $\wh{\Delta}^\bullet_{k}$.

\begin{thm}\label{prespseq}
Suppose $\cc F^\bullet$
   \begin{equation}\label{complex}
    \cdots\xrightarrow{d^{-3}}\cc F^{-2}\xrightarrow{d^{-2}}\cc F^{-1}\xrightarrow{d^{-1}}\cc F^0\to 0\to\cdots
   \end{equation}
is a cochain complex of rationally contractible presheaves
concentrated in non-positive degrees. Let $\cc L^{-n}:=\kr d^{-n}$, $n>0$,
and $\cc L:=\coker d^{-1}$. Then the chain complex of Abelian groups
$\cc L(\wh{\Delta}^\bullet_{k})$ is zig-zag quasi-isomorphic to the chain
complex $\cc L^{-1}(\wh{\Delta}^\bullet_{k})[-2]$. Moreover, there is a tower in the derived category
$\cc D(\Ab)$ of chain complexes of Abelian groups which are
concentrated in non-positive degrees
   \begin{equation}\label{tower}
     \cdots\lra{\alpha^{-3}}\cc L^{-3}(\wh{\Delta}^\bullet_{k})[-2]\lra{\alpha^{-2}}\cc L^{-2}(\wh{\Delta}^\bullet_{k})[-1]
     \lra{\alpha^{-1}}\cc L^{-1}(\wh{\Delta}^\bullet_{k})
   \end{equation}
with $q$-th layer, $q\geq 0$, being the complex $\cc
H^{-1-q}(\wh{\Delta}^\bullet_{k})[-q]$. In
particular, the tower~\eqref{tower} gives rise to a strongly
convergent spectral sequence
   $$E^2_{pq}=H_{p}(\cc H^{-1-q}(\wh{\Delta}^\bullet_{k})):=H_{p}(\cc H^{-1-q}(\tau_{<0}\cc F^\bullet)(\wh{\Delta}^\bullet_{k}))
       \Rightarrow H_{p+q+2}(\cc L(\wh{\Delta}^\bullet_{k})).$$
\end{thm}

\begin{proof}
Consider a short exact sequence of presheaves
   $$0\to\im d^{-1}\to\cc F^0\to\cc L\to 0.$$
It induces a short exact sequence of Abelian groups in each degree
$n\geq 0$:
   $$0\to\lp_{U\ni v_0,\ldots,v_n}(\im d^{-1})(U)\to\lp_{U\ni v_0,\ldots,v_n}(\cc F^0)(U)\to\lp_{U\ni v_0,\ldots,v_n}\cc L(U)\to 0,$$
where $v_0,\ldots,v_n$ are the vertices of $\Delta^n_k$ and $U\subseteq\Delta^n_k$. We also use here the fact that
the direct limit functor is exact.
The latter is nothing but the short exact sequence
   $$0\to(\im d^{-1})(\wh{\Delta}^n_{k})\to(\cc F^0)(\wh{\Delta}^n_{k})\to\cc L(\wh{\Delta}^n_{k})\to 0.$$
Since $\cc F^0$ is rationally contractible by assumption, it follows
that the complex of Abelian groups $(\cc
F^0)(\wh{\Delta}^\bullet_{k})$ is contractible by
Proposition~\ref{facts}(2). Now the induced triangle in $\cc D(\Ab)$
   $$(\im d^{-1})(\wh{\Delta}^\bullet_{k})\to(\cc F^0)(\wh{\Delta}^\bullet_{k})
       \to\cc L(\wh{\Delta}^\bullet_{k})\lra{\tau}(\im d^{-1})(\wh{\Delta}^\bullet_{k})[-1]$$
yields a zig-zag quasi-isomorphism of chain complexes $\tau:\cc
L(\wh{\Delta}^\bullet_{k})\lra{\sim}(\im
d^{-1})(\wh{\Delta}^\bullet_{k})[-1]$. For the same reasons, $(\im
d^{-1})(\wh{\Delta}^\bullet_{k})$ is zig-zag quasi-isomorphic to $\cc
L^{-1}(\wh{\Delta}^\bullet_{k})[-1]$. For this, one uses the short
exact sequence of presheaves $\cc L^{-1}\hookrightarrow\cc
F^{-1}\twoheadrightarrow\im d^{-1}$. So $\cc
L(\wh{\Delta}^\bullet_{k})$ is zig-zag quasi-isomorphic to $\cc
L^{-1}(\wh{\Delta}^\bullet_{k})[-2]$.

Similarly, each short exact sequence of presheaves
   $$0\to\cc L^{-n}\lra{i_n}\cc F^{-n}\lra{p_n}\im d^{-n}\to 0$$
gives rise to a quasi-isomorphism of chain complexes 
$a_n:(\im d^{-n})(\wh{\Delta}^\bullet_{k})\simeq\cc L^{-n}(\wh{\Delta}^\bullet_{k})[-1]$.
It is induced by the map 
$b_n:\im d^{-n}\to\cc L^{-n}[-1]$ in $\cc D(\Ab)$ given by the zig-zag map of chain complexes
   $$\xymatrix{&\cc L^{-n}\ar[d]_{i_n}\ar[r]^\id&\cc L^{-n}\\
                       \im d^{-n}&\cc F^{-n}\ar[l]_{p_n}}$$
Here the chain complex in the middle is concentrated in degrees 1 and 0.

Next, each short exact sequence of presheaves
   $$0\to\im d^{-n-1}\lra{j_n}\cc L^{-n}\to\cc H^{-n}\to 0$$
gives rise to a triangle in $\cc D(\Ab)$
   $$\cc L^{-n-1}(\wh{\Delta}^\bullet_{k})[-1]\xrightarrow{}\cc L^{-n}(\wh{\Delta}^\bullet_{k})
       \to\cc H^{-n}(\wh{\Delta}^\bullet_{k})\to\cc L^{-n-1}(\wh{\Delta}^\bullet_{k}).$$
In this way we obtain the desired tower~\eqref{tower} with layers as
stated. Up to shift the morphism $\alpha^{-n}$ equals the composite map $j_n(\wh{\Delta}_k^\bullet)\circ a_n^{-1}$. 
Note that the $n$th complex of the tower $\cc
L^{-n-1}(\wh{\Delta}^\bullet_{k})[-n]$ is $(n-1)$-connected, and
hence the tower gives rise to a strongly convergent spectral
sequence
   $$E_2^{pq}=H_{p}(\cc H^{-1-q}(\wh{\Delta}^\bullet_{k}))\Rightarrow H_{p+q+2}(\cc L(\wh{\Delta}^\bullet_{k}))$$
after applying~\cite[6.1.1]{FS} to it. This completes the proof.
\end{proof}

\begin{rem}
Similarly to~\cite[6.1]{FS} the spectral sequence of the preceding theorem 
occurs from an exact couple (with maps $i,j,k$ of bidegrees $(1,-1),(0,0),
(-2,1)$ respectively) --- see~\cite[p.~798]{FS} for details.
\end{rem}

Let $\cc A$ be a $V$-category of correspondences on $Sm/k$ in the
sense of~\cite{GG} ($V$-categories are just a formal abstraction of
basic properties for the category of finite correspondences $Cor$).
We say that $\cc A$ is {\it nice\/} if for any smooth semilocal
scheme $W$ and any $\bb A^1$-invariant presheaf $\cc F$ with $\cc
A$-correspondences the canonical morphism of presheaves $\cc F\to\cc
F_{\zar}$ induces an isomorphism of Abelian groups $\cc
F(W)\lra{\cong}\cc F_{\zar}(W)$. Here $\cc F_{\zar}$ is the Zariski
sheaf associated to the presheaf $\cc F$. For example, the category
of finite correspondences $Cor$ is nice by~\cite[4.24]{Voe}. If the
base field $k$ is infinite perfect of characteristic different from
2, then the category of finite $MW$-correspondences in the sense
of~\cite{CF} is nice by~\cite[3.5]{BF}.

\begin{cor}\label{shvspseq}
Under the conditions of Theorem~\ref{prespseq} suppose
that~\eqref{complex} is a cochain complex of Zariski sheaves with
nice correspondences such that its presheaves $\cc L$ and $\cc
H^{-q}$-s are $\bb A^1$-invariant. Then the chain complex of Abelian
groups $\cc L_{\zar}(\wh{\Delta}^\bullet_{k})$ is quasi-isomorphic
to the chain complex of Abelian groups $\cc L^{-1}(\wh{\Delta}^\bullet_{k})[-2]=\cc
L^{-1}_{\zar}(\wh{\Delta}^\bullet_{k})[-2]$. Moreover, the $q$-th layer of the tower~\eqref{tower}
equals the complex $\cc
H^{-1-q}_{\zar}(\wh{\Delta}^\bullet_{k})[-q]$. Here $\cc
H^{-q}_{\zar}$ stands for the $-q^{\textrm th}$ cohomology Zariski
sheaf of the complex of Zariski sheaves~\eqref{complex}. In
particular, the tower~\eqref{tower} gives rise to a strongly
convergent spectral sequence
   $$E^2_{pq}=H_{p}(\cc H^{-1-q}_{\zar}(\wh{\Delta}^\bullet_{k})):=H_{p}(\cc H^{-1-q}_{\zar}(\tau_{<0}\cc F^\bullet)(\wh{\Delta}^\bullet_{k}))
       \Rightarrow H_{p+q+2}(\cc L_{\zar}(\wh{\Delta}^\bullet_{k})).$$
\end{cor}

\section{Semilocal Milnor $K$-theory}

Let $\bb Z(n)$ be Suslin--Voevodsky's motivic complex of Zariski sheaves of weight
$n\geq 0$ on $Sm/k$ (see~\cite[Definition~3.1]{SV}). By definition, it is concentrated in
cohomological degrees $m\leq n$. More precisely, it equals the cochain complex with differential (of degree $+1$)
equal to the alternating sum of face operations
   \begin{equation}\label{Z(n)}
    \cdots\to Cor(\Delta^2_k\times-,\bb G_m^{\wedge n})\to Cor(\Delta^1_k\times-,\bb G_m^{\wedge n})\to Cor(-,\bb G_m^{\wedge n})\to 0\to\cdots
   \end{equation}
Here the Zariski sheaf $Cor(-,\bb G_m^{\wedge n})$ is in cohomological degree $n$. 
Basing on well-known theorems of Nesterenko--Suslin~\cite{NS} and Totaro~\cite{Tot},
define the {\it $n$-th Milnor $K$-theory sheaf\/} $\cc K_n^M$ as
the Zariski sheaf $\cc H^n_{\zar}(\bb Z(n))$.

Similarly, let $\wt{\bb Z}(n)$ be Calm\`es--Fasel's Milnor--Witt motivic
complex of Zariski sheaves of weight $n\geq 0$ on $Sm/k$ (see~\cite{CF}). More precisely, it equals the cochain
complex with differential (of degree $+1$) equal to the alternating sum of face operations
   \begin{equation}\label{tildeZ(n)}
    \cdots\to \wt{Cor}(\Delta^2_k\times-,\bb G_m^{\wedge n})\to\wt{Cor}(\Delta^1_k\times-,\bb G_m^{\wedge n})
    \to\wt{Cor}(-,\bb G_m^{\wedge n})\to 0\to\cdots
   \end{equation}
Denote by $\cc K_n^{MW}$ the Zariski sheaf $\cc H^n_{\zar}(\wt{\bb
Z}(n))$. We shall also refer to $\cc K_n^{MW}$ as the {\it $n$-th
Milnor--Witt $K$-theory sheaf}.

\begin{dfn}\label{maindef}
Let $k$ be any field and $n\geq 0$. The {\it $n$-th semilocal Milnor
$K$-theory complex of the field $k$\/} is the chain complex of
Abelian groups ${\cc K}_n^M(\wh{\Delta}^\bullet_{k})$.

The {\it $(n,q)$-th semilocal Milnor $K$-theory group
$\wh{K}_{n,q}^M(k)$ of $k$\/} is defined as the $q$-th homology
group $H_q({\cc K}_n^M(\wh{\Delta}^\bullet_{k}))$ of the $n$-th
semilocal Milnor $K$-theory complex of $k$. By definition,
$\wh{K}_{n,q}^M(k)=0$ for all $q<0$.

Let $k$ be an infinite perfect field of characteristic not 2 and
$n\geq 0$. The {\it $n$-th semilocal Milnor--Witt $K$-theory complex
of the field $k$\/} is the chain complex of Abelian groups ${\cc
K}_n^{MW}(\wh{\Delta}^\bullet_{k})$.

The {\it $(n,q)$-th semilocal Milnor--Witt $K$-theory group
$\wh{K}_{n,q}^{MW}(k)$ of $k$\/} is defined as the $q$-th homology
group $H_q({\cc K}_n^{MW}(\wh{\Delta}^\bullet_{k}))$ of the $n$-th
semilocal Milnor $K$-theory complex of $k$. By definition,
$\wh{K}_{n,q}^{MW}(k)=0$ for all $q<0$.

If $A$ is an Abelian group then the same definitions are given ``with
$A$-coefficients", in which case we just tensor the relevant complexes by $A$ to get
${\cc K}_n^M(\wh{\Delta}^\bullet_{k})\otimes A$ and ${\cc K}_n^{MW}(\wh{\Delta}^\bullet_{k})\otimes A$
and then semilocal Milnor and Milnor--Witt $K$-theory groups with $A$-coefficients
$\wh{K}_{n,*}^M(k,A)$, $\wh{K}_{n,*}^{MW}(k,A)$
are homology groups of these complexes. In what follows we mostly deal with the case $A=\bb Q$,
in which case we write the subscript~$\bb Q$. 

All statements which are proven below
with integer coefficients will automatically be true with $\bb Q$-coefficients. The interested reader 
will always be able to repeat the relevant proofs rationally (we do not write them for brevity).
Also, many statements are valid with any coefficients, say, when $A$ is finite. Since motivic 
cohomology with finite coefficients is well studied, we do not discuss this case either,
assuming that the interested reader will do this easily.
\end{dfn}

We also recall from~\cite{EVMS,Kerz,NS} that the $n$th Milnor $K$-group
$K^M_n(R)$ of a commutative ring $R$ is the abelian group generated
by symbols $\{a_1,\ldots,a_n\}$, $a_i\in R^\times$, $i =
1,\ldots,n$, subject to the following relations:

\begin{enumerate}
\item for any $i$, $\{a_1,\ldots,a_ia'_i,\ldots,a_n\}=\{a_1,\ldots,a_i,\ldots,a_n\}+\{a_1,\ldots,a'_i,\ldots,a_n\}$;
\item $\{a_1,\ldots,a_n\}=0$ if there exist $i,j$, $i\ne j$, such that $a_i+a_j =0$ or 1.
\end{enumerate}

\begin{rem}\label{naive}
If the field $k$ is infinite, it follows from~\cite{EVMS,Kerz} that
${\cc K}_n^M(\wh{\Delta}^\ell_{k})=K_n^M(\cc O(\ell)_{k,v})$. 
The $n$-th semilocal Milnor $K$-theory chain complex ${\cc
K}_n^M(\wh{\Delta}^\bullet_{k})$ is therefore isomorphic to the chain complex
$K_n^M(\cc O(\bullet)_{k,v})$. In particular,
$\wh{K}_{n,q}^{M}(k)=H_q(K_n^M(\cc O(\bullet)_{k,v}))$ for all
$n,q\geq 0$. We see that 
${\cc K}_n^M(\wh{\Delta}^\ell_{k})$ is defined naively in terms of generators and relations.
\end{rem}

\begin{lem}\label{weight01}
Given any field $k$, the complex ${\cc K}_0^M(\wh{\Delta}^\bullet_{k})$
has only one non-zero homology group in degree zero,
$\wh{K}_{0,0}^M(k)$, which is isomorphic to $\bb Z$.
\end{lem}

\begin{proof}
This follows from the fact that the complex $\bb Z(0)$ is canonically quasi-isomorphic to the constant
sheaf $\bb Z$, positioned in degree 0 (see~\cite[3.2]{SV}).
\end{proof}

Recall from~\cite[2.3.1]{KS} that a presheaf with transfers $\cc F$
is {\it birationally invariant\/} if $\cc F(X)\lra{\simeq}\cc F(U)$ for any
dense open immersion $j:U\to X$. Birationally invariant homotopy
invariant presheaves with transfers are called {\it birational
sheaves}.

\begin{lem}\label{birat}
Given a birational sheaf $\cc F$ and an irreducible $X\in Sm/k$, the natural
morphism of complexes of Abelian groups
   $$\cc F(k(X))\to\cc F(\wh{\Delta}^\bullet_{k(X)})$$
is a quasi-isomorphism.
\end{lem}

\begin{proof}
By~\cite[4.1.3]{KL} $\cc F$ is a birational motivic sheaf.
By~\cite[p.~513]{KL} $\cc F(k(X))\to\cc
F(\wh{\Delta}^\bullet_{k(X)})$ is a quasi-isomorphism.
\end{proof}

We are now in a position to prove the main result of the paper.

\begin{thm}\label{glavnaya}
Suppose $k$ is any field. The following statements are true:

$(1)$ For any $n\geq 1$, $\wh{K}_{n,0}^M(k)=\wh{K}_{n,1}^M(k)=0$.

$(2)$ For any $n\geq 1$, there is a strongly convergent spectral
sequence
   $$E^2_{pq}=H_{p}(\cc H^{n-1-q}_{\zar}(\tau_{<n}\bb Z(n))(\wh{\Delta}^\bullet_{k}))
       %=H_{p}(H^{n-1-q}_{\zar}(\wh{\Delta}^\bullet_{k},\tau_{<n}\bb Z(n)))
       \Longrightarrow\wh{K}_{n,p+q+2}^M(k),$$
where $\tau_{<n}\bb Z(n)$ is the truncation complex of $\bb Z(n)$ for degrees smaller than $n$.

$(3)$ If $n=2$ and $p\leq 1$, there is an isomorphism of Abelian
groups $H^p_{\zar}(k,\bb Z(2))=\wh{K}_{2,3-p}^M(k)$.

$(4)$ If the field $k$ is infinite perfect of characteristic
different from 2 and $n\geq 1$, then the natural morphism of chain
complexes of Abelian groups ${\cc
K}_n^{MW}(\wh{\Delta}^\bullet_{k})\to{\cc
K}_n^{M}(\wh{\Delta}^\bullet_{k})$ is a quasi-isomorphism. In
particular, it induces isomorphisms of Abelian groups
$\wh{K}_{n,q}^{MW}(k)\lra{\cong}\wh{K}_{n,q}^{M}(k)$ for all
$q\in\bb Z$.
\end{thm}

\begin{proof}
(1)-(2). By~Example~\ref{primery} the Zariski sheaf with transfers
$Cor(-,\bb G_m^{\wedge n})$, $n\geq 1$, is rationally contractible.
It follows from Proposition~\ref{facts} that the Zariski sheaf with
transfers $Cor(\Delta^\ell_k\times-,\bb G_m^{\wedge n})$ is
rationally contractible for every $\ell\geq 0$. Now the desired
spectral sequence of the second assertion follows from
Corollary~\ref{shvspseq} if we apply it to the cochain complex of
rationally contractible Zariski sheaves with transfers~\eqref{Z(n)}.
Another argument proving (2) is that there is an exact triangle
   $$0\to Tot(\tau_{<n}\bb Z(n)(\wh{\Delta}^\bullet_{k}))\to
       Tot(\bb Z(n)(\wh{\Delta}^\bullet_{k}))\to K^M_n(\wh{\Delta}^\bullet_{k})[-n]\to 0$$
using the fact that the complex in the middle is acyclic.

Next, by Theorem~\ref{prespseq} and Corollary~\ref{shvspseq} there is a
quasi-isomorphism ${\cc K}_n^M(\wh{\Delta}^\bullet_{k})\simeq\cc
L^{-1}_{\zar}(\wh{\Delta}^\bullet_{k})[-2]$ of chain complexes for
some presheaf with transfers $\cc L^{-1}$ (the shift is
homological). Assertion~(1) now follows.

(3). Suppose $k$ is perfect. Since the motivic complex of
weight one $\bb Z(1)$ is acyclic in non-positive cohomological
degrees by~\cite[3.2]{SV}, the proof of~\cite[4.34]{Voe} and
Voevodsky's Cancellation Theorem~\cite{V1} imply that
   $$\underline{\Hom}(\bb G_m^{\wedge 1},\cc H^p_{\zar}(\bb Z(2)))\cong(\cc H^p(\underline{\Hom}(\bb G_m^{\wedge 1},\bb Z(2))))_{\zar}=0$$
for all $p\leq 1$. It follows from~\cite[2.5.2]{KS} that each $\cc H^p_{\zar}(\bb
Z(2))$, $p\leq 1$, is a birational (Nisnevich) sheaf.

Lemma~\ref{birat} implies that the natural map of chain complexes
   $$\cc H^p_{\zar}(\bb Z(2))(k)\to\cc H^p_{\zar}(\bb Z(2))(\wh{\Delta}^\bullet_{k}),\quad p\leq 1,$$
is a quasi-isomorphism. Thus $H_i(\cc H^p_{\zar}(\bb
Z(2))(\wh{\Delta}^\bullet_{k}))=0$ for all $i\ne 0$ and $H_0(\cc
H^p_{\zar}(\bb Z(2))(\wh{\Delta}^\bullet_{k}))=\cc H^p_{\zar}(\bb
Z(2))(k)$. Using Corollary~\ref{shvspseq} applied to  the cochain complex of
rationally contractible Zariski sheaves with transfers~\eqref{Z(n)}, the $q$th layer $\cc
H^{-1-q}_{\zar}(\wh{\Delta}^\bullet_{k})[-q]$ of
the tower~\eqref{tower} is a complex having only one homology group
in degree $q$. Using induction in $q$ and applying homology functor
to each triangle $\cc
L^{-q-2}_{\zar}(\wh{\Delta}^\bullet_{k})[-q-1]\xrightarrow{}\cc
L^{-q-1}_{\zar}(\wh{\Delta}^\bullet_{k})[-q]\to\cc
H^{-1-q}_{\zar}(\wh{\Delta}^\bullet_{k})[-q]\lra{+}$ coming from the tower~\eqref{tower}, we get an
isomorphism of Abelian groups
%$H_q(\cc
%L^{-1}_{\zar}(\wh{\Delta}^\bullet_{k}))=H_0(\cc
%L^{-q-1}_{\zar}(\wh{\Delta}^\bullet_{k}))=\cc
%H^{-1-q}_{\zar}({k})=H^{-1-q}_{\zar}(k,\bb Z(2))$ for $q\geq 0$.
$H^p_{\zar}(k,\bb Z(2))=\wh{K}_{2,3-p}^M(k)$ for any $p\leq 1$.

%We see that the strongly convergent spectral sequence from the
%second assertion collapses for any perfect field $k$.

%Note that by using Suslin's results~\cite{Suslin17} the same arguments can be applied to non-perfect fields
%for proving our assertion after inverting the characteristic. However, the assertion is valid with integer coefficients
%and for any field as it is shown below.

Next, suppose $K/k$ is a finitely generated field extension of the
perfect field $k$. Then $K=k(U)$ for some $U\in Sm/k$. Each scheme
$\wh{\Delta}^\ell_{K}$ is the semilocalization of
$\Delta^\ell_k\times U$ at the points
$(v_0,\eta),\ldots,(v_n,\eta)$, where $\eta$
is the generic point of $U$. %(see the proof of~\cite[3.10]{BF})
Lemma~\ref{birat} implies that the natural map of chain complexes
   $$\cc H^p_{\zar}(\bb Z(2))(K)\to\cc H^p_{\zar}(\bb Z(2))(\wh{\Delta}^\bullet_{K}),\quad p\leq 1,$$
is a quasi-isomorphism. Thus $H_i(\cc H^p_{\zar}(\bb
Z(2))(\wh{\Delta}^\bullet_{K}))=0$ for all $i\ne 0$ and $H_0(\cc
H^p_{\zar}(\bb Z(2))(\wh{\Delta}^\bullet_{K}))=\cc H^p_{\zar}(\bb
Z(2))(K)$. As above, we get isomorphisms of Abelian groups
$H^p_{\zar}(K,\bb Z(2))=\wh{K}_{2,3-p}^M(K)$, $p\leq 1$, for any
finitely generated field extension $K/k$.

%We see that the strongly convergent spectral sequence from the
%second assertion also collapses for any finitely generated field
%extension $K/k$.

Finally, for any field $K$ of characteristic $p$ we follow the same argument as in~\cite[p.~244]{Suslin}.
We use the fact that it can be written as a directed limit $K=\lp_i K_i$ of fields finitely generated over
$\bb Z/p$, and the fact that the above homology/cohomology groups
commute with directed limits (we use~\cite[Lemma 3.9]{MVW}). We also use here the fact that
the cohomology groups $H^*_{\zar}(K,\bb Z(n))=\lp H^*_{\zar}(K_i,\bb Z(n))$ are defined
intrinsically in terms of the field $K$ and are independent of the
choice of the base field.
%We also refer the reader to~\cite[Section~2]{BF}
%for rational contractibility of presheaves over field extensions extending
%the relevant facts of~\cite[Section~2]{Suslin}.

(4). The proof is based on Bachmann's results on the $MW$-motivic
cohomology~\cite{Bachmann}. By Example~\ref{primery} the Zariski
sheaf $\wt{Cor}(-,\bb G_m^{\wedge n})$ is rationally contractible
for every $n>0$.

If we consider the cochain complex of Zariski
sheaves~\eqref{tildeZ(n)} and repeat the arguments for the proof of
the second assertion, we shall get a strongly convergent spectral
sequence
   $$\wt{E}^2_{pq}:=H_{p}(\cc H^{n-1-q}_{\zar}(\wt{\bb Z}(n))(\wh{\Delta}^\bullet_{k}))
       %=H_{p}(H^{n-1-q}_{\zar}(\wh{\Delta}^\bullet_{k},\wt{\bb Z}(n)))
       \Longrightarrow\wh{K}_{n,p+q+2}^{MW}(k).$$
The natural functor of additive categories of correspondences
$\wt{Cor}\to Cor$ induces a map of spectral sequences
$\wt{E}^2_{pq}\to{E}^2_{pq}$, where $E^2_{pq}$ is the spectral
sequence of the second assertion.

It follows from~\cite[Theorem~17]{Bachmann} that the morphism of
complexes $\tau_{<n}\wt{\bb Z}(n)\to\tau_{<n}\bb Z(n)$ is a
quasi-isomorphism, locally in the Nisnevich topology. This means
that each morphism of Nisnevich sheaves $\cc H_{\nis}^p(\wt{\bb
Z}(n)_{\nis})\to\cc H_{\nis}^p({\bb Z}(n))$, $p\ne n$, is an
isomorphism. The relation between notations of~\cite{Bachmann} and
this paper is as follows: $\pi_i(\wt{H}\bb Z)_j=\cc
H_{\nis}^{j-i}(\wt{\bb Z}(j)_{\nis})$, $\pi_i({H_\mu}\bb Z)_j=\cc
H_{\nis}^{j-i}({\bb Z}(j)_{\nis})$.

It follows from~\cite[3.5]{BF} that
   $$\cc H_{\nis}^p(\wt{\bb Z}(n)_{\nis})(\wh{\Delta}^\ell_{k})=\cc H_{\zar}^p(\wt{\bb Z}(n))(\wh{\Delta}^\ell_{k}),\quad\ell\geq 0.$$
Using~\cite[5.5]{Voe}, one has
   $$\cc H_{\nis}^p({\bb Z}(n))(\wh{\Delta}^\ell_{k})=\cc H_{\zar}^p({\bb Z}(n))(\wh{\Delta}^\ell_{k}),\quad\ell\geq 0.$$

Therefore the morphism of strongly convergent spectral sequences
$\wt{E}^2_{pq}\to{E}^2_{pq}$ is an isomorphism. This isomorphism
implies that the map of chain complexes ${\cc
K}_n^{MW}(\wh{\Delta}^\bullet_{k})\to{\cc
K}_n^{M}(\wh{\Delta}^\bullet_{k})$ is a quasi-iso\-mor\-phism, as was to
be proved.
\end{proof}

%\begin{rem}
%This remark is to indicate the interested reader what facts precisely %(modulo simple technicalities) 
%are necessary to prove 
%the preceding theorem (except item~(4) about semilocal Milnor--Witt $K$-theory and facts established in the previous sections). 
%These are various Voevodsky's theorems for homotopy invariant presheaves
%with transfers~\cite{Voe}, Cancellation theorem~\cite{V1}, Suslin's results for rationally 
%contractible presheaves~\cite{Suslin} as well as the Gysin triangle for motives~\cite{Voe1}.
%\end{rem}

\begin{cor}\label{weight1}
Given any field $k$, semilocal Milnor $K$-theory complex ${\cc
K}_1^M(\wh{\Delta}^\bullet_{k})$ is acyclic. In particular,
$\wh{K}_{1,q}^M(k)=0$ for all $q\in\bb Z$.
\end{cor}

\begin{proof}
By~\cite[3.2]{SV} the complex $\tau_{<1}\bb Z(1)$ is acyclic.
Therefore the $E^2$-page of the strongly convergent spectral
sequence of Theorem~\ref{glavnaya}(2) for $n=1$ is trivial. Our
statement now follows.
\end{proof}

\begin{cor}\label{coker}
$\wh{K}_{n,2}^M(k)=\coker(H^{n-1}_{\zar}(\wh{\Delta}^1_k,\bb Z(n))\xrightarrow{\partial_1-\partial_0}
H^{n-1}_{\zar}(k,\bb Z(n)))$ for any field $k$ and $n>1$.
\end{cor}

If $k$ is a field of positive characteristic $p$, then 
Milnor $K$-theory groups of $k$ are $p$-torsionfree by a theorem of Izhboldin~\cite{Izh}. 
The following statement says that
semilocal Milnor $K$-theory groups are $p$-uniquely divisible.

\begin{cor}\label{poschar}
Given any field $k$ of positive characteristic $p>0$, semilocal Milnor $K$-theory groups 
$\wh{K}_{n,m}^M(k)$ are $p$-uniquely divisible for all $n>0$ and $m\in\bb Z$. In particular,
$\wh{K}_{n,m}^M(k)=\wh{K}_{n,m}^M(k)\otimes\bb Z[1/p]$.
\end{cor}

\begin{proof}
We claim that $\cc H^s_{\zar}(\bb Z(n))$ is a sheaf of $\bb Z[1/p]$-modules for all $n>0$ and $s<n$. The 
Geisser--Levine theorem~\cite[1.1]{GeiLev} implies that $\cc H^s_{\zar}(\bb Z(n))(K)$ is $p$-uniquely divisible 
for any field extension $K/k$. It follows that the morphism of sheaves 
$\cc H^s_{\zar}(\bb Z(n))\to\cc H^s_{\zar}(\bb Z(n))\otimes\bb Z[1/p]$ is an isomorphism on field extensions $K/k$,
and hence it is an isomorphism of sheaves by~\cite[4.20]{Voe}.

We see that the $E^2$-term of the strongly convergent spectral sequence of Theorem~\ref{glavnaya}(2) consists
of $p$-uniquely divisible Abelian groups, and hence the semilocal Milnor $K$-theory groups of the 
statement are $p$-uniquely divisible.
\end{proof}

\begin{rem}
The preceding theorem implies that the evaluation of the
Milnor--Witt sheaf ${\cc K}_n^{MW}$, $n\geq 1$, at
$\wh{\Delta}^\bullet_{k}$ ``deletes" the information about quadratic
forms.
\end{rem}

By Remark~\ref{naive} if the base field $k$ is infinite, Milnor
$K$-theory of semilocal schemes like $\wh{\Delta}^\ell_{k}$ has an
explicit, naive description, whereas motivic cohomology involves
sophisticated constructions. Thus Theorem~\ref{glavnaya} computes
some motivic cohomology groups as homology groups of certain naive
complexes. In particular, we can apply ``symbolic" computations to cycles in 
motivic complexes. It is worth pointing out that the term
``symbolic" here refers to the symbols in the definition of Milnor
$K$-theory.
%(it is not readily apparent that this is amenable to
%symbolic computation implementable on a computer).

Furthermore, a theorem of Kerz~\cite[1.2]{Kerz} implies that
the norm residue homomorphism induces an isomorphism of complexes
   $$\cc K^M_n(\wh{\Delta}^\bullet_k)\otimes\bb Z/\ell\lra{\cong}H^n_{et}(\wh{\Delta}^\bullet_k,\mu_\ell^{\otimes n}),\quad n>0,$$
if the field $k$ is infinite of characteristic not dividing $\ell$. 
Since the second and the third statement of Theorem~\ref{glavnaya}
are true with finite coefficients, it follows that semilocal Milnor $K$-theory groups with
finite coefficients can be computed as homology groups of complexes $H^n_{et}(\wh{\Delta}^\bullet_k,\mu_\ell^{\otimes n})$.

Recall from~\cite{Suslin90} that the indecomposable $K_3$-group of a
field $k$, denoted by $K_3^{\ind}(k)$, is defined as the cokernel of
the canonical homomorphism $K_3^M(k)\to K_3(k)$.

\begin{cor}\label{Kind}
For any field $k$, there is an isomorphism $K_3^{\ind}(k)=\wh{K}_{2,2}^{M}(k)$.
In particular, $K_{3}(k)^{(2)}_{\bb Q}=\wh{K}_{2,2}^{M}(k)_{\bb Q}$.
\end{cor}

\begin{proof}
The motivic spectral sequence gives an isomorphism
$K_3^{\ind}(k)=H^1_{\zar}(k,\bb Z(2))$. Now Theorem~\ref{glavnaya}
implies the claim.
\end{proof}

\begin{cor}\label{weight2}
For any perfect field $k$, any connected $X\in Sm/k$ and any $p\leq
1$, there is an isomorphism
   $$H^p_{\zar}(X,\bb Z(2))=\wh{K}_{2,3-p}^{M}(k(X)),$$
where $k(X)$ is the function field of $X$.
\end{cor}

\begin{proof}
This follows from Lemma~\ref{birat}, the proof of Theorem~\ref{glavnaya}(3) and the fact that for any
$p\leq 1$ there is an isomorphism of groups $H^p_{\zar}(X,\bb
Z(2))=H^p_{\zar}(k(X),\bb Z(2))$.
\end{proof}

Since $H^p_{\zar}(k,\bb Z(q))$ is uniquely divisible for $p\leq 0$ and any field $k$ (see, e.g.,~\cite[Exercise~VI.4.6]{Wei}),
Theorem~\ref{glavnaya}(3) implies the following

\begin{cor}\label{uniqdiv}
For any field $k$ and any $n\geq 3$, the group
$\wh{K}_{2,n}^{M}(k)$ is uniquely divisible.
\end{cor}

\begin{cor}\label{ratKtheory}
For any field $k$ there are isomorphisms of rational vector spaces
   $$K_{4+p}(k)^{(2)}_{\bb Q}=\wh{K}_{2,3+p}^{M}(k),\quad p\geq 0.$$
Moreover, $K_3(k)_{\bb Q}={K}_{3}^{M}(k)_{\bb Q}\oplus\wh{K}_{2,2}^{M}(k)_{\bb Q}$.
\end{cor}

\begin{proof}
This follows from Theorem~\ref{glavnaya}, Corollary~\ref{uniqdiv} and the fact that for any
$p\geq -1$ there is an isomorphism $H^{-p}_{\zar}(k,\bb Q(2))=K_{4+p}(k)^{(2)}_{\bb Q}$.
\end{proof}

\begin{thm}\label{trancs}
Semilocal Milnor $K$-theory is invariant under purely
transcendental extensions. Na\-me\-ly,
$\wh{K}_{n,q}^{M}(k)=\wh{K}_{n,q}^{M}(k(x))$ for any field $k$ and
$n,q\geq 0$.
\end{thm}

\begin{proof}
Suppose the base field $k$ is perfect. Then the Zariski sheaf $\cc K_n^M$ on $Sm/k$ is strictly homotopy invariant.
Its Eilenberg--Mac~Lane motivic $S^1$-spectrum $EM(\cc K_n^M)$ is $\bb A^1$-local by~\cite[6.2.2]{Mor}, hence
$H^n_{\nis}(X,\cc K_n^M)=SH_{S^1}(k)(X_+,EM(\cc K_n^M)[n])$, where $SH_{S^1}(k)$ is
the homotopy category of motivic $S^1$-spectra. In particular,
the homology groups of the complex $\cc K_n^M(\wh{\Delta}_{k(X)/k}^\bullet)$ are iso\-morphic to 
the homotopy groups of the spectrum $EM(\cc K_n^M)(\wh{\Delta}_{k(X)/k}^\bullet)$
for any irreducible $X\in Sm/k$.
By~\cite[2.2.6]{KL} the latter spectrum is stably weakly equivalent to $s_0(EM(\cc K_n^M))(X)$, where $s_0(EM(\cc K_n^M))$
is the zeroth slice of $EM(\cc K_n^M)$ in the motivic Postnikov tower (see~\cite[\S1.2]{KL}).
Applying $s_0(EM(\cc K_n^M))$ to the morphism $\bb A^1\to pt$, one gets a stable weak equivalence 
of spectra $s_0(EM(\cc K_n^M))(pt)\to s_0(EM(\cc K_n^M))(\bb A^1)$. We see that
$\wh{K}_{n,q}^{M}(k)=\wh{K}_{n,q}^{M}(k(x))$ for any $n,q\geq 0$.

Suppose now $k$ is any field of characteristic $p>0$. Denote by $k_\infty$ the perfect closure of $k$.
Following Suslin~\cite{Suslin17}, the theory of motivic cohomology and associated categories of motives
with $p^{-1}$-coefficients over $k$ is essentially the same with the theory over $k_\infty$. We shall write
$\cc K_{n,\infty}^M$ for the $n$th Milnor $K$-theory sheaf on $Sm/k_\infty$. By~\cite[1.1 and 1.3]{Suslin17}
the canonical functor $\phi^\sharp:Shv_{\nis}(k)\to Shv_{\nis}(k_\infty)$ between the categories of Nisnevich sheaves
takes $\cc K_{n}^M$ to $\cc K_{n,\infty}^M$ (they are Nisnevich sheaves as well).
It follows from~\cite[2.13]{Suslin17} that the natural morphism of chain complexes
$\cc K_{n}^M(\wh{\Delta}_{k}^\bullet)[p^{-1}]\to\cc K_{n,\infty}^M(\wh{\Delta}_{k_\infty}^\bullet)[p^{-1}]$
is an isomorphism. It remains to apply Corollary~\ref{poschar} and invariance under purely
transcendental extensions for perfect fields proven above.
%Let $K/k$ is a finitely generated field extension of $k$. There is $X\in Sm/k$ such that $K=k(X)$.
%It follows from~\cite[2.2.6 and 4.2.1]{KL} that the natural map of chain complexes 
%$\cc K_n^M(\wh{\Delta}^\bullet_K)\to\cc K_n^M(\wh{\Delta}^\bullet_{K(x)})$ induced by the 
%projection $X\times\bb A^1\to X$ is a quasi-isomorphism. We tacitly use here~\cite[4.7]{Suslin17} as well.
%We have thus shown the theorem for finitely generated field extensions $K/k$.
%For any field of characteristic $p$ we use the fact
%that it can be written as a direct limit of fields finitely
%generated over $\bb Z/p$, and the fact that semilocal Milnor $K$-theory groups
%obviously commute with direct limits.
\end{proof}

We say that two smooth $k$-varieties $X$ and $Y$ are {\it stably $\bb A^1$-equivalent\/} if their
suspension motivic $S^1$-spectra $\Sigma_{S^1}^\infty X_+$ and $\Sigma_{S^1}^\infty Y_+$ are
isomorphic in the homotopy category of motivic $S^1$-spectra $SH_{S^1}(k)$. If $Y=pt$ we call $X$ 
{\it stably $\bb A^1$-contractible}. For example, any motivic equivalence between $X$ and $Y$ 
in the category of motivic spaces induces
an isomorphism of $\Sigma_{S^1}^\infty X_+$ and $\Sigma_{S^1}^\infty Y_+$.

The following result says that semilocal Milnor $K$-theory is an invariant for stably $\bb A^1$-equivalent varieties.

\begin{cor}\label{neozhidal}
Suppose $k$ is a perfect field. If $X,Y\in Sm/k$ are irreducible and stably $\bb A^1$-equivalent, then 
$\wh{K}_{n,q}^{M}(k(X))$ is isomorphic to $\wh{K}_{n,q}^{M}(k(Y))$ for any $n,q\geq 0$.
In particular, if $X$ is irreducible and stably $\bb A^1$-contractible, then $\wh{K}_{n,q}^{M}(k(X))$ is isomorphic to 
$\wh{K}_{n,q}^{M}(k)$ for any $n,q\geq 0$.
\end{cor}

\begin{proof}
We use the proof of Theorem~\ref{trancs} showing that the $\wh{K}_{n,q}^{M}(k(X))$ (respectively, $\wh{K}_{n,q}^{M}(k(Y))$)
are isomorphic to homotopy groups $\pi_*(s_0(EM(\cc K_n^M))(X))=SH_{S^1}(k)(\Sigma_{S^1}^\infty X_+[*],s_0(EM(\cc K_n^M)))$
(respectively, $\pi_*(s_0(EM(\cc K_n^M))(X))=SH_{S^1}(k)(\Sigma_{S^1}^\infty Y_+[*],s_0(EM(\cc K_n^M)))$).
\end{proof}

\section{Some criteria for the Beilinson--Soul\'e Vanishing Conjecture}

In this section an application of the technique developed in the
previous sections is given. Recall that the Beilinson--Soul\'e
Vanishing Conjecture states that each
%$K_{2n-p}(k)^{(n)}_{\bb Q}$ vanish for all fields $k$, $q>0,p\leq 0$ (see~\cite{Lev93}).
%Equivalently, $H^p_{\zar}(k,\bb Q(q))=0$ for any field $k$ and
%$q>0,p\leq 0$. One of its equivalent formulations says that each
complex $\bb Z(n)$, $n>0$, on $Sm/k$ is acyclic outside the interval
of cohomological degrees $[1,n]$.
%Equivalently, for any $X\in Sm/k$
%and any point $x\in X$ the complex of Abelian groups $\bb
%Z(n)(\spec(\cc O_{X,x}))$ is acyclic outside the interval $[1,n]$.
It follows from~\cite[p.~352]{V02} that it suffices to verify
acyclicity of the complex $\bb Z(n)$ on $Sm/k$ outside the interval
$[1,n]$ whenever $k$ is perfect. Therefore the base field $k$ is
assumed to be perfect throughout this section.
%It is the case for $\bb Q(1)$
%(see~\cite[3.2]{SV}), but completely unclear for higher weights
%$n\geq 2$.

The main result of this section, Theorem~\ref{bscriteria}, gives
equivalent conditions for the Beilinson--Soul\'e Vanishing
Conjecture. In particular, it says that instead of verifying
acyclicity of the sophisticated complexes $\bb Z(n)$ in non-positive
cohomological degrees, it is enough to verify acyclicity of the
chain complexes of Abelian
groups $\cc Z^0(n)(\wh{\Delta}_{K/k}^\bullet)$, 
where $\cc Z^0(n):=\kr\partial^0$ with $\partial^{0}:Cor(\Delta^{n}_k\times-,\gmpn)\to Cor(\Delta^{n-1}_k\times-,\gmpn)$ 
being the zeroth differential of the complex $\bb Z(n)$, and $K/k$ is a
finitely generated field extension.

\begin{thm}\label{bscriteria}
The following conditions are equivalent:
\begin{itemize}
\item[$(1)$] the Beilinson--Soul\'e Vanishing Conjecture is true for complexes $\bb Z(n)$, $n\geq 1$, on $Sm/k$;
\item[$(2)$] for every $n\geq 1$ and every finitely generated field extension $K/k$,
$\cc Z^0(n)(\wh{\Delta}_{K/k}^\bullet)$ is an acyclic chain complex;
\item[$(3)$] for every $n\geq 1$, $\cc Z^0(n)$ has a resolution in the
category of Zariski sheaves
   $$\cdots\lra{d^{-2}}\cc R^{-1}\lra{d^{-1}}\cc R^0\to\cc Z^0(n)$$
such that each $\cc R^i$ is rationally contractible and cohomology
presheaves $\cc H^{i<0}$ are trivial on the semilocal schemes of the
form $\wh{\Delta}^\ell_{K/k}$, where $K/k$ is a finitely generated
field extension;
\item[$(4)$] for every $n\geq 1$ and every finitely generated field extension
$K/k$, the total complex of the bicomplex $\tau_{\leq 0}(\bb
Z(n))(\wh{\Delta}_{K/k}^\bullet)$ is acyclic, where $\tau_{\leq 0}$
is the truncation in corresponding cohomological degrees;
\item[$(5)$] for every $n\geq 1$ and every finitely generated field extension $K/k$,
the total complex of the bicomplex $\tau_{[1,n]}(\bb
Z(n))(\wh{\Delta}_{K/k}^\bullet)$ is acyclic.
\end{itemize}
\end{thm}

\begin{proof}
$(1)\Rightarrow(2)$. Given $\ell\geq 0$ and $n>0$, set
   $$\cc W^{-1-\ell}(n):=\kr(\partial^{-\ell}:Cor(\Delta^{n+\ell}_k\times-,\gmpn)\to Cor(\Delta^{n+\ell-1}_k\times-,\gmpn)).$$
The proof of Theorem~\ref{prespseq} and Corollary~\ref{shvspseq}
shows that there is a tower in the derived category $\cc D(\Ab)$ of
chain complexes of Abelian groups
   \begin{equation}\label{towerr}
     \cdots\lra{\alpha^{-3}}\cc W^{-3}(n)(\wh{\Delta}^\bullet_{K/k})[-2]\lra{\alpha^{-2}}
     \cc W^{-2}(n)(\wh{\Delta}^\bullet_{K/k})[-1]
     \lra{\alpha^{-1}}\cc W^{-1}(n)(\wh{\Delta}^\bullet_{K/k})
   \end{equation}
with $q$-th layer, $q\geq 0$, being the complex $\cc
H^{-q}_{\zar}(\wh{\Delta}^\bullet_{K/k})[-q]$. Here $\cc
H^{-q}_{\zar}$ stands for the $-q^{\textrm{th}}$ cohomology sheaf of
the complex $\bb Z(n)$. By definition, $\cc W^{-1}(n)=\cc Z^{0}(n)$.
%The complexes have an index shift compared to 
%the tower~\eqref{tower} as the entries of~\eqref{towerr} cosists of $\cc Z^{-\ell}(n)=\kr \partial^{-\ell}$
%for non-positive $\ell$ including $\ell=0$.
We use here the fact that
$Cor(\Delta^{n+\ell}_k\times-,\gmpn)$ is a rationally contractible
sheaf by~Example~\ref{primery} and Proposition~\ref{facts}. Similarly to Theorem~\ref{glavnaya} the
tower~\eqref{towerr} yields a strongly convergent spectral sequence
   \begin{equation}\label{spspsp}
    E^2_{pq}=H_{p+q}(\cc H^{-q}_{\zar}(\wh{\Delta}_{K/k}^\bullet))
     \Rightarrow H_{p+q}(\cc Z^{0}(n)(\wh{\Delta}_{K/k}^\bullet))
   \end{equation}
By assumption, $\cc H^p_{\zar}(\bb Z(n))=0$ for $p\leq 0$. Therefore
the strongly convergent spectral sequence~\eqref{spspsp} is trivial,
and hence $\cc Z^{0}(n)(\wh{\Delta}_{K/k}^\bullet)$ is acyclic.

$(2)\Rightarrow(1)$. We use induction in $n$. By~\cite[3.2]{SV} $\bb Z(1)$ is acyclic in non-positive degrees,
hence the base case $n=1$.
Suppose $\bb Z(n-1)$ is acyclic outside the
interval $[1,n-1]$. We want to show that $\bb Z(n)$ is acyclic outside the
interval $[1,n]$.
%Since $k$ is perfect,
%and $\bb Z(0)$ is a complex concentrated in degree zero,
Voevodsky's Cancellation Theorem~\cite{V1} together with~\cite[4.34]{Voe} implies that
   $$\underline{\Hom}(\bb G_m^{\wedge 1},\cc H^p_{\zar}(\bb Z(n)))
     =\underline{\Hom}(\bb G_m^{\wedge 1},\cc H^p_{\nis}(\bb Z(n)))=
     \cc H^{p-1}_{\nis}(\bb Z(n-1))=0$$
for all $p\leq 0$. We use here the fact that $\cc F_{\zar}=\cc
F_{\nis}$ for any homotopy invariant presheaf with transfers
(see~\cite[5.5]{Voe}). It follows from~\cite[2.5.2]{KS} that each
$\cc H^p_{\zar}(\bb Z(n))$, $p\leq 0$, is a birational (Nisnevich) sheaf.

Let $K/k$ be a finitely generated field extension. Then $K=k(X)$ for
some $X\in Sm/k$. Lemma~\ref{birat} implies that the natural map of chain complexes
   $$\cc H^p_{\zar}(\bb Z(n))(K)\to\cc H^p_{\zar}(\bb Z(n))(\wh{\Delta}^\bullet_{K}),\quad p\leq 0,$$
is a quasi-isomorphism. Thus $H_i(\cc H^p_{\zar}(\bb
Z(n))(\wh{\Delta}^\bullet_{K}))=0$ for all $i\ne 0$ and $H_0(\cc
H^p_{\zar}(\bb Z(n))(\wh{\Delta}^\bullet_{K}))=\cc H^p_{\zar}(\bb
Z(n))(K)$. Therefore the strongly convergent spectral
sequence~\eqref{spspsp} collapses, and hence
   $$0=H_i(\cc Z^{0}(n)(\wh{\Delta}_{K/k}^\bullet))=\cc H^{-i}_{\zar}(\bb Z(n))(K),\quad i\geq 0.$$
Each sheaf $\cc H^{-i}_{\zar}(\bb Z(n))$ is homotopy invariant
by~\cite{Voe} and trivial on finitely generated field extensions. It
follows from~\cite[4.20]{Voe} that $\cc H^{-i}_{\zar}(\bb Z(n))=0$,
hence $\bb Z(n)$ is acyclic outside the interval $[1,n]$. 

$(1)\Rightarrow(3)$. This is straightforward: set $\cc
R^\ell:=Cor(\Delta^{n+\ell+1}_k\times-,\gmpn)$ with differentials
being those of $\bb Z(n)$. We use here the facts that
$Cor(\Delta^{n+\ell}_k\times-,\gmpn)$ is rationally
contractible (see Example~\ref{primery} and Proposition~\ref{facts}) and that for any smooth
semilocal scheme $W$, any $\bb A^1$-invariant presheaf $\cc F$ with
transfers the canonical morphism $\cc F(W)\lra{\cong}\cc
F_{\zar}(W)$ is an isomorphism~\cite[4.24]{Voe}.

$(3)\Rightarrow(2)$. This follows from the spectral sequence of Theorem~\ref{prespseq}.

$(1)\Rightarrow(4)$. This is obvious.

$(4)\Leftrightarrow(5)$. It is enough to observe that the total
complex of the bicomplex $(\bb Z(n))(\wh{\Delta}_{K/k}^\bullet)$ is
acyclic for $n>0$. The latter easily follows from~\cite[2.2;
2.4]{Suslin} (see~\cite[2.3]{BF} as well).

$(4)\Rightarrow(2)$. The complex $\tau_{\leq 0}(\bb Z(n))$ equals
   $$\cdots\to Cor(\Delta^{n+2}_k\times-,\bb G_m^{\wedge n})\to Cor(\Delta^{n+1}_k\times-,\bb G_m^{\wedge n})\to\cc Z^0(n)\to 0\to\cdots$$
Example~\ref{primery}(1), Proposition~\ref{facts} and~\cite[4.7]{Suslin17} imply that the complex
$Cor(\Delta^{n+\ell}_k\times-,\bb G_m^{\wedge n})(\wh{\Delta}_{K/k}^\bullet)$ is acyclic for all $n>0$.
The spectral sequence for a double complex implies that the homology groups 
of the complex $\tau_{\leq 0}(\bb Z(n))(\wh{\Delta}_{K/k}^\bullet)$ are those of the complex
$\cc Z^0(n)(\wh{\Delta}_{K/k}^\bullet)$, and hence the remaining implication follows.
\end{proof}

%\begin{rem}\label{z(1)}
%The proof of the preceding theorem gives an alternative proof for a
%result of Suslin--Voevodsky~\cite[3.2]{SV} saying that $\bb Z(1)$ is
%acyclic outside degree $n=1$ (the fact that $\cc H^1_{\zar}(\bb
%Z(1))=\cc O^{\times}$ is easy -- see, e.g.,~\cite[3.2.0]{SV}).
%\end{rem}

\section{$K_4$ of a field}

In this section another application of semilocal Milnor $K$-theory is given. 
We show that the group $K_4(k)$ is completely determined by extensions involving the classical 
Milnor $K$-theory and semilocal Milnor $K$-theory. If $k$ is algebraically
closed then $K_4(k)$ is a direct sum of relevant Milnor $K$-theory and semilocal Milnor $K$-theory
groups of $k$. 
%Throughout this section the base field $k$ is perfect unless otherwise specified.

Recall that the motivic spectral sequence relates algebraic $K$-theory to motivic cohomology~\cite{FS}
   \begin{equation}\label{mss}
    E_{p,q}^2=H_{\zar}^{q-p}(k,\bb Z(q))\Rightarrow K_{p+q}(k).
   \end{equation}
It is a strongly convergent spectral sequence concentrated in the first quadrant. It is obtained from a tower
of connected $S^1$-spectra
   \begin{equation}\label{msstower}
    \cdots\to\bb K^3\to\bb K^2\to\bb K^1\to\bb K^0:=K(k),
   \end{equation}
where $K(k)$ is Quillen's $K$-theory spectrum of $k$.
Rationally, the motivic spectral sequence collapses at
$E^2=E^\infty$ and
   \begin{equation}\label{mssq}
    K_n(k)_{\bb Q}=\bigoplus_{q}H^{2q-n}_{\zar}(k,\bb Q(q))
   \end{equation}
(see~\cite{FW} for details).

\begin{prop}\label{fdelta}
Let $k$ be a perfect field (respectively any field with $\charr(k)=p>0$) and let
$\cc F$ be a homotopy invariant Nisnevich sheaf with transfers of 
Abelian groups (respectively $\bb Z[1/p]$-modules). Let
$\wh{\bb A}^1_k$ be the
semilocalization of the affine line at $0,1$. Then
   $$\cc F(\wh{\bb A}^1_k)=\cc F(k)\bigoplus\left(\bigoplus_{x\in\bb A^1_k\setminus\{0,1\}}\cc F_{-1}(k(x))\right),$$
where each $x$ in the direct sum is a closed point of $\bb A^1_k$.
\end{prop}

\begin{proof}
Suppose $k$ is perfect and $U$ is an open subset of $\bb A^1_k$ with $Z=\bb A^1_k\setminus U=\{x_1,\ldots,x_n\}$. The 
Gysin triangle for motives~\cite{Voe1} gives a triangle in $DM^{eff}(k)$
   $$\bigoplus_1^n M(\bb G_{m,k(x_i)}^{\wedge1})\to M(U)\to M(pt)\lra{+}$$
If $x_0\in U$ is a rational point, then this triangle splits. If we apply $DM^{eff}(k)(-,\cc F)$
to this triangle, one gets a canonical isomorphism
$\cc F(U)=\cc F(k)\oplus(\oplus_{i=1}^n\cc F_{-1}(k(x_i)))$. It follows that
   $$\cc F(\wh{\bb A}^1_k)=\colim_{U\ni\{0,1\}}\cc F(U)=\cc F(k)\bigoplus\left(\bigoplus_{x\in\bb A^1_k\setminus\{0,1\}}\cc F_{-1}(k(x))\right),$$
as required. Here the splitting onto the first summand is given by $x_0:=0\in\bb A^1_k$.

The statement for fields of positive characteristic is the same if we use Suslin's results~\cite{Suslin17}
saying that Voevodsky's theory for motivic complexes works for non-perfect fields as well provided that we deal with
sheaves with transfers of $\bb Z[1/p]$-modules.
\end{proof}

The following result says that the motivic cohomology groups $H^{n-1}_{\zar}(k,\bb Z(n))$ fit into a finite tower
of homomorphisms of groups with subsequent quotients being
semilocal Milnor $K$-theory groups.

\begin{thm}\label{hn-1n}
There are exact sequences of Abelian groups
   $$\bigoplus_{x\in\bb A^1_k\setminus\{0,1\}}H^{n-2}_{\zar}(k(x),\bb Z(n-1))\xrightarrow{u} H^{n-1}_{\zar}(k,\bb Z(n))\to\wh{K}^M_{n,2}(k)\to 0$$
and
   $$\bigoplus_{x\in\bb A^1_k\setminus\{0,1\}}\wh{K}^M_{2,2}(k(x))\xrightarrow{u}H^{2}_{\zar}(k,\bb Z(3))\to\wh{K}^M_{3,2}(k)\to 0.$$
Here $n>1$ and each $x$ of the left direct sums is a closed point.
The homomorphism $u$ is the restriction of 
   $$\partial_1-\partial_0:H^{n-1}_{\zar}(\wh{\bb A}^1_k,\bb Z(n))\to H^{n-1}_{\zar}(k,\bb Z(n))$$
to $\bigoplus_{x\in\bb A^1_k\setminus\{0,1\}}H^{n-2}_{\zar}(k(x),\bb Z(n-1))\subset H^{n-1}_{\zar}(\wh{\bb A}^1_k,\bb Z(n))$.
\end{thm}

\begin{proof}
The first exact sequence follows from Proposition~\ref{fdelta} and Corollary~\ref{coker}.
%and Cancellation Theorem~\cite{V1}.
%We also use here the proof of Corollary~\ref{poschar} showing that $\cc H_{\zar}^{n-1}(\bb Z(n))$ is a sheaf of $\bb Z[1/p]$-modules
%whenever $\charr(k)=p>0$.
%Here the homomorphism $u$ is the restriction of 
%   $$\partial_1-\partial_0:H^{n-1}_{\zar}(\wh{\bb A}^1_k,\bb Z(n))\to H^{n-1}_{\zar}(k,\bb Z(n))$$
%to $\bigoplus_{x\in\bb A^1_k\setminus\{0,1\}}H^{n-2}_{\zar}(k(x),\bb Z(n-1))\subset H^{n-1}_{\zar}(\wh{\bb A}^1_k,\bb Z(n))$.
The second exact sequence is a particular case of the first one if we apply Theorem~\ref{glavnaya}(3).
\end{proof}

\begin{cor}\label{corhn-1n}
Let $n>1$ and $\cc X=\{\wh{K}^M_{\ell,2}(k(x))\mid \textrm{ $x$ is a closed point in $\bb A^1_k$ and $2\leq \ell\leq n$}\}$. Then
$H^{n-1}_{\zar}(k,\bb Z(n))$ belongs to the smallest localizing Serre subcategory of $\Ab$ containing $\cc X$.
\end{cor}

\begin{proof}
This is a consequence of Theorem~\ref{hn-1n} and~\cite[Proposition~2]{GG09}.
\end{proof}

The motivic spectral sequence~\eqref{mss} gives a long exact sequence of Abelian groups
   \begin{multline*}
    H^{-1}_{\zar}(k,\bb Z(2))\to\pi_1(\bb K^3)\to K_4(k)\to H^0_{\zar}(k,\bb Z(2))\xrightarrow{d} K_3^M(k)\to K_3(k)\to H^1_{\zar}(k,\bb Z(2))\to 0.
   \end{multline*}
Here $\bb K^3$ is the fourth entry of the tower~\eqref{msstower}. It follows from~\cite[VI.4.3.2]{Wei} that $d=0$.
By Corollary~\ref{weight2}
the latter long exact sequence can be rewritten as 
   \begin{multline*}
    \wh{K}^M_{2,4}(k)\to\pi_1(\bb K^3)\to K_4(k)\to \wh{K}^M_{2,3}(k)\xrightarrow{0}K_3^M(k)\to K_3(k)\to \wh{K}^M_{2,2}(k)\to 0.
   \end{multline*}
Next, by using the motivic spectral sequence we find that $\pi_1(\bb K^3)$ fits into an exact sequence
   $$K^M_4(k)\to\pi_1(\bb K^3)\to H^2_{\zar}(k,\bb Z(3))\to 0.$$
By Theorem~\ref{hn-1n} $H^2_{\zar}(k,\bb Z(3))$ fits into an exact sequence
   $$\bigoplus_{x\in\bb A^1_k\setminus\{0,1\}}\wh{K}^M_{2,2}(k(x))\xrightarrow{u}H^{2}_{\zar}(k,\bb Z(3))\to\wh{K}^M_{3,2}(k)\to 0,$$
where each $x$ in the direct sum is a closed point of $\bb A^1_k$.

We see that $\pi_1(\bb K^3)$ is expressed in terms of Milnor $K$-theory and semilocal Milnor $K$-theory groups,
and hence so is $K_4(k)$.

We are now in a position to prove the main result of the section.

\begin{thm}\label{K4}
Let $k$ be any field. The following statements are true:

$(1)$ $K_4(k)$ is entirely expressed in terms of Milnor $K$-theory and semilocal Milnor $K$-theory groups. Precisely,
$K_4(k)$ fits into an exact sequence 
   $$\wh{K}^M_{2,4}(k)\to A\to K_4(k)\to \wh{K}^M_{2,3}(k)\to 0,$$
where $A$ is an Abelian group fitted into an exact sequence 
   $$K^M_4(k)\to A\to B\to 0$$
with $B$ fitted into an exact sequence
   $$\bigoplus_{x\in\bb A^1_k\setminus\{0,1\}}\wh{K}^M_{2,2}(k(x))\to B\to\wh{K}^M_{3,2}(k)\to 0.$$
%If $k$ is any field of exponential characteristic $e$ then the same statement is true 
%with $\bb Z[1/e]$-coefficients. 

$(2)$ There is an isomorphism of Abelian groups
   $$K_4(k)_{\bb Q}\cong K_4^M(k)_{\bb Q}\oplus\wh{K}^M_{3,2}(k)_{\bb Q}\oplus\wh{K}^M_{2,3}(k)\oplus F,$$
where $F$ is a direct summand of $\bigoplus_{x\in\bb A^1_k\setminus\{0,1\}}\wh{K}^M_{2,2}(k(x))_{\bb Q}$.

$(3)$ If $k$ is algebraically closed then there is an isomorphism of Abelian groups
   $$K_4(k)\cong K_4^M(k)\oplus\wh{K}^M_{3,2}(k)_{\bb Q}\oplus\wh{K}^M_{2,3}(k)\oplus F,$$
where $F$ is a direct summand of $\bigoplus_{k^\times\setminus \{1\}}\wh{K}^M_{2,2}(k)_{\bb Q}$.

The isomorphisms from $(2)$ and $(3)$ are not canonical.
\end{thm}

\begin{proof}
The first statement follows from the arguments above the theorem. 
%whenever $k$ is perfect.
%In turn, if $k$ is any field of exponential characteristic $e$ then we use~\cite{Suslin17} to
%verify the same assertion with $\bb Z[1/e]$-coefficients. 
The second statement is a consequence
of the first statement and isomorphism~\eqref{mssq}. It also uses a rational splitting of the exact sequence 
for $H^2_{\zar}(k,\bb Z(3))$ from Theorem~\ref{hn-1n}.
%(we work here with rational coefficients, hence
%we use~\cite{Suslin17} to verify the assertion for non-perfect fields). 
We also use here the fact that the group $\wh{K}^M_{2,3}(k)$ is uniquely divisible by Corollary~\ref{uniqdiv}. Finally, 
the third statement follows from the second one and the 
fact that $K_4^M(k)$ and $K_4(k)$ are uniquely divisible Abelian groups if $k$ 
is algebraically closed (see, e.g.,~\cite[pp.~267, 511, 514]{Wei}).
\end{proof}

It is worth mentioning that since $H^{-2}_{\zar}(k,\bb Z(1))=0$ by~\cite[3.2]{SV},
the Beilinson--Soul\'e Vanishing Conjecture for $K_4$ requires only $\wh{K}^M_{2,3}(k)=0$.

\section{On conjectures of Beilinson and Parshin}

Let $k$ be a field of characteristic $p>0$. A conjecture of Beilinson~\cite[2.4.2.2]{Bei} says 
that Milnor $K$-theory and Quillen $K$-theory agree rationally:
   $$K_n^M(k)_{\bb Q}\lra{\cong} K_n(k)_{\bb Q}.$$
The purpose of this section is to show that the Beilinson conjecture is equivalent to vanishing of the 
rational semilocal Milnor $K$-theory. Since semilocal Milnor $K$-theory is defined in elementary terms,
its vanishing with rational coefficients should be much easier for verification than the original Beilinson conjecture.
We shall also show in this section that vanishing of the 
rational semilocal Milnor $K$-theory is a necessary condition for Parshin's conjecture.

\begin{thm}\label{beilinson}
The Beilinson conjecture for rational algebraic $K$-theory of fields of positive characteristic is true if and only if
rational semilocal Milnor $K$-theory groups $\wh{K}_{n,m}^M(k)_{\bb Q}$ of such fields vanish for all $n>0, m\geq 0$.
\end{thm}

\begin{proof}
Assume the Beilinson conjecture. Then the isomorphism~\eqref{mssq} implies $H^i_{\zar}(k,\bb Q(n))=0$
for $i\ne n$ and all fields of prime characteristic. It follows that Zariski cohomology sheaves except the 
$n$th cohomology are zero (we use here~\cite[4.20]{Voe}). Now the spectral sequence of Theorem~\ref{glavnaya}
and Corollary~\ref{weight1} imply $\wh{K}_{n,m}^M(k)_{\bb Q}$ vanish for all $n>0, m\geq 0$.

Conversely, suppose that $\wh{K}_{n,m}^M(k)_{\bb Q}$ vanish for all $n>0, m\geq 0$
and all fields of prime characteristic. We claim that each complex $\bb Q(n)$, $n\geq 1$, 
has only one non-zero cohomology sheaf in degree $n$. We use induction 
in $n$. By~\cite[3.2]{SV} $\bb Q(1)$ is acyclic in non-positive degrees,
hence the base case $n=1$.

%By Theorem~\ref{glavnaya}(3)
%   $$0=\wh{K}_{2,3-m}^M(k)_{\bb Q}=H^{m}(k,\bb Q(2)),\quad m\leq 1.$$
%As above, it follows that all Zariski cohomology sheaves of $\bb Q(2)$ are zero except the
%second cohomology sheaf.

Assume that the complex $\bb Q(n)$, $n\geq 1$, has only one non-zero cohomology sheaf in degree $n$.
Repeating the proof of Theorem~\ref{glavnaya}(3) word for word (see the proof of $(2)\Rightarrow(1)$ in Theorem~\ref{bscriteria} as well), 
we obtain that %each cohomology sheaf $\cc H^{i}_{\zar}(\bb Q(n+1))$, $i\ne n+1$, is birational and
   $$H^{m}_{\zar}(k,\bb Q(n+1))=\wh{K}_{n+1,n-m+2}^M(k)_{\bb Q}=0,\quad m\leq n.$$
Then Zariski cohomology sheaves except the 
$(n+1)$th cohomology of $\bb Q(n+1)$ are zero (we use here~\cite[4.20]{Voe}),
and our claim follows. 

The isomorphism~\eqref{mssq} now implies that the natural homomorphism
$K_n^M(k)_{\bb Q}\to K_n(k)_{\bb Q}$ is an isomorphism, as was to be shown.
\end{proof}

\begin{rem}
(1) We should stress that in the proof of Theorem~\ref{beilinson} we work with all fields of prime characteristic
(not with individual ones) in order to annihilate the relevant Zariski cohomology sheaves.

(2) By Corollary~\ref{poschar}, $\wh{K}_{n,m}^M(k)=\wh{K}_{n,m}^M(k)\otimes\bb Z[1/p]$ for all $n>0, m\geq 0$.
It follows that $\wh{K}_{n,m}^M(k)_{\bb Q}=\wh{K}_{n,m}^M(k)\otimes\bb Z_{(p)}$. Therefore, $\wh{K}_{n,m}^M(k)_{\bb Q}=0$
if and only if $\wh{K}_{n,m}^M(k)\otimes\bb Z_{(p)}=0$.
\end{rem}

Recall that Parshin's conjecture~\cite[2.4.2.3]{Bei} states that for any smooth projective variety $X$ defined over a finite field, 
the higher algebraic $K$-groups vanish rationally:
   $$K_i(X)_{\bb Q}=0,\quad i>0.$$

We finish the section by the following
   
\begin{thm}\label{parshin}
Let $k$ be a field of characteristic $p>0$ and assume Parshin's conjecture. Then
rational semilocal Milnor $K$-theory groups $\wh{K}_{n,m}(k)_{\bb Q}$ vanish for all $n>0, m\geq 0$.
\end{thm}

\begin{proof}
It follows from~\cite[p.~203]{Gei} that $H^i_{\zar}(k,\bb Q(n))=0$ for $i\ne n$. The proof of Theorem~\ref{beilinson}
shows that rational semilocal Milnor $K$-theory groups $\wh{K}_{n,m}(k)_{\bb Q}$ vanish for all $n>0, m\geq 0$,
as required.
\end{proof}

\section{Concluding remarks}\label{concl}

%The celebrated Milnor conjecture proven by Voevodsky~\cite{V02b} states that the norm
%residue homomorphism
%   $$K_*^M(k)/2\to H^*_{et}(k,\mu_\ell^*)$$
%is an isomorphism for fields of characteristic different from 2. 
In contrast with motivic cohomology with mod 2 coefficients, we show in this section that semilocal Milnor
$K$-theory groups with $\mathbb Z/2$-coefficients $\wh{K}^M_{n,*}(k,\mathbb Z/2)$ 
are zero for any $n>1$ (see Definition~\ref{maindef}).
This is another property of semilocal Milnor $K$-theory distinguishing it with the classical Milnor 
$K$-theory/motivic cohomology of fields. This also distinguishes semilocal Milnor $K$-theory with
relative Milnor $K$-theory in the sense of Levine~\cite{L92}.
More precisely, we have the following

\begin{thm}\label{mod2}
For any infinite perfect field $k$ and any $n>1$ the $n$-th semilocal Milnor $K$-theory complex with $\mathbb Z/2$-coefficients
$\cc K_n^M(\wh{\Delta}_k^\bullet)\otimes\bb Z/2$ is acyclic or, equivalently, the semilocal Milnor
$K$-theory groups with $\mathbb Z/2$-coefficients $\wh{K}^M_{n,*}(k,\mathbb Z/2)$ 
are zero.
\end{thm}

\begin{proof}
We separate two cases in the proof: when $\charr(k)=2$ and $\charr(k)\ne 2$. If $\charr(k)=2$ then Geisser--Levine's 
theorem~\cite{GeiLev} implies that the sheaf $\cc K_n^M/2$ is quasi-isomorphic to a shift of 
$\mathbb Z/2(n)$. It follows from Example~\ref{primery} and Proposition~\ref{facts} that the total complex of the bicomplex 
$\mathbb Z/2(n)(\wh{\Delta}_k^\bullet)$ is acyclic, hence so is the complex $\cc K_n^M(\wh{\Delta}_k^\bullet)\otimes\bb Z/2$.

Suppose now that $\charr(k)\ne 2$. Set,
   $$I^{n+1}:=\ker(\cc K_n^{MW}\to\cc K^M_n),\quad n\geq 0,$$
where $\cc K_n^{MW},\cc K^M_n$ are Nisnevich sheaves of Milnor--Witt and Milnor
$K$-theory respectively. 
%There exists a pullback diagram of homotopy invariant sheaves
%   $$\xymatrix{\cc K_n^{MW}\ar[d]\ar[r]&\cc K^M_n\ar[d]\\
%                       I^n\ar[r]&\cc K^M_n/2,}$$
%in which all arrows are epimorphisms of sheaves. 
By~\cite[7.10]{Kerz} and~\cite[6.3]{GSZ} there is an isomorphism of sheaves
   \begin{equation}\label{dr}
    I^n/I^{n+1}=\cc K^M_n/2.
   \end{equation}
Consider a short exact sequence of sheaves
   $$0\to I^{n+1}\to\cc K_n^{MW}\to\cc K^M_n\to 0,\quad n\geq 1.$$
It follows from Theorem~\ref{glavnaya}(4) and~\cite[3.6]{BF} that the complex $I^{n+1}(\wh{\Delta}_k^\bullet)$ 
is acyclic for all $n\geq 1$. The isomorphism~\eqref{dr} and~\cite[3.6]{BF} imply that the complex 
$\cc K_n^M(\wh{\Delta}_k^\bullet)\otimes\bb Z/2$ is acyclic if $n>1$.
\end{proof}

In characteristic $p$, it follows from the Geisser--Levine theorem~\cite{GeiLev} that the logarithmic de 
Rham--Witt sheaf $W_r\Omega^n_{\log}$ is quasi-isomorphic to a shift of 
$\mathbb Z/p^r(n)$. Using this quasi-isomorphism, the proof of the preceding theorem shows that the complex 
$W_r\Omega^n_{\log}(\wh{\Delta}_k^\bullet)$ is acyclic. The converse is also true. More precisely,
using the technique developed in this paper for semilocal Milnor $K$-theory one can show that if 
the complex $W_1\Omega^n_{\log}(\wh{\Delta}_k^\bullet)$ is acyclic for any $n>0$ then the complex
$\mathbb Z/p(n)$ has only one cohomology sheaf isomorphic to $W_1\Omega^n_{\log}$. 
The latter implies (using induction in $r$) that the only non-trivial cohomology sheaf of $\mathbb Z/p^r(n)$
is $W_r\Omega^n_{\log}$.

The above arguments justify to raise the following

\begin{conj}
Each logarithmic de 
Rham--Witt sheaf $W_r\Omega^n_{\log}$ is rationally contractible.
\end{conj}

This conjecture will shed new light not only on the fundamental theorem of Geisser--Levine~\cite{GeiLev}
but also on further properties of the logarithmic de 
Rham--Witt sheaf $W_r\Omega^n_{\log}$ which is of fundamental importance. In particular, if the conjecture
were solved in the affirmative then the complex $W_r\Omega^n_{\log}(\wh{\Delta}_k^\bullet)$ would be contractible 
by Proposition~\ref{facts}(2). A good starting point for the conjecture would be rational contractibility of $W_1\Omega^n_{\log}$.

\end{document}